\documentclass[draft]{publmathdeb}
\usepackage{amsmath,amsfonts,amssymb}

\usepackage{mathtools}

\DeclareMathOperator{\supp}{supp}

\author{Nithi Rungtanapirom}
\address{Department of Mathematics and Computer Science, Faculty of Science\\ Chulalongkorn University\\ 10330 Bangkok\\ Thailand}
\email{nithi.r@chula.ac.th}

\author{J\"orn Steuding}
\address{Department of Mathematics\\ W\"urzburg University\\ Emil Fischer-Str. 40, 97074 W\"urzburg\\ Germany} 
\email{joern.steuding@uni-wuerzburg.de}

\author{Saeree Wananiyakul}
\address{Department of Mathematics and Computer Science, Faculty of Science\\ Chulalongkorn University\\ 10330 Bangkok\\ Thailand}
\email{s.wananiyakul@hotmail.com}

\title{Effective weak universality in short intervals}



\submitted{}

\theoremstyle{plain}
\newtheorem{theorem}{Theorem}
\newtheorem{prop}[theorem]{Proposition}

\theoremstyle{definition}

\theoremstyle{remark}

\begin{document}
\begin{abstract}
We prove an effective multidimensional $\Omega$-result of Voronin in short intervals $[T,T+H]$ with $T^{\frac{27}{82}}\le H\le T$ and derive an effective weak universality theorem of the Riemann zeta-function.  Furthermore, we provide a few remarks for these results under the assumption of the Riemann hypothesis.
\end{abstract}

\maketitle

\section{Statement of the Main Results}

In 1975, Sergei Voronin \cite{vor} proved his celebrated universality theorem for the Riemann zeta-function $\zeta$ which states that, roughly speaking, every non-vanishing analytic function $g$, defined on a disk $D\coloneqq\{s:\vert s\vert\leq r\}$ of sufficiently small radius~$r$, can be uniformly approximated by shifts $\zeta(s+i\tau)$ of the Riemann zeta-function. Moreover, given $\varepsilon>0$, the set of real $\tau>0$ satisfying 
\[\max_{s\in D}\bigg\vert \zeta\Big(s+\frac 34+i\tau\Big)-g(s)\bigg\vert<\varepsilon\]
has a positive lower density. Two years earlier, Voronin \cite{voroninj} had proved a multi-dimensional denseness theorem which continued earlier work by Harald Bohr and may be considered as a first step towards the later universality theorem (cf. \cite{glmss}). In 1988, Voronin \cite{voronin1989} obtained an {\it effective} version of the latter result which implies a weak effective version of universality as shown by Ram\= unas Garunk\v stis et al. \cite{glmss} (and which was probably known to Voronin as well). Here {\it effective} means that there exists an explicit upper bound for the shift $\tau$. 

In this note, we extend these results to {\it short} intervals $[T,T+H]$ for the shifts~$\tau$, a~concept that was recently introduced by Antanas Laurin\v cikas \cite{laurincikas}. Note that {\it short} here means that $H=o(T)$. For treating short intervals, let $N(\alpha;T,H)$ be the number of zeros $\beta+i\gamma$ of $\zeta(s)$ in the rectangle $\alpha<\beta$ and $T\le t\le T+H$ (counting multiplicities). We shall use an advanced density theorem due to Ramachandran Balasubramanian \cite{balasubramanian1978}, which is stated that
\begin{equation}\label{eq: bal}
N(\alpha;T,H)\ll H^{\frac{4(1-\alpha)}{3-2\alpha}}(\log H)^{100},
\end{equation}
uniformly in $\alpha$ and $H$, for all $H$ satisfying $T^{\frac{27}{82}}\le H\le T$. 

Our main results are the following:

\begin{theorem}\label{Thm 1} 
Let $N\in \mathbb{N}$, $\sigma_0\in (\frac 12,1),\boldsymbol{a}=(a_0,a_1,\dots,a_{N-1})\in \mathbb C^N$ and $\varepsilon\in (0,1)$ be arbitrary but fixed. Then, the system of inequalities
\[
\bigg|\frac{d^k}{d s^k}\log \zeta(s)\Big|_{s=\sigma_0+ i\tau}-a_k\bigg|<\varepsilon \qquad\mbox{for}\quad k=0,1,\ldots,N-1
\]
has a solution $\tau\in [T,T+H]$ provided that $T^{\frac{27}{82}}\le H\le T$ and
\[
T\ge \exp_2\bigg(C_{1}(N,\sigma_0)\Big(\|\boldsymbol{a}\|+\frac 1{\varepsilon}\Big)^{\frac 8{1-\sigma_0}+{\frac{8}{\sigma_0-{\frac 1 2}}}}\bigg),
\]
where $C_1(N,\sigma_0)$ is a positive, effectively computable constant depending only on $N,\sigma_0$ {(and not $H$!)}, and $\|\boldsymbol{a}\|\coloneqq\sum_{0 \le k\le N-1}|a_k|$.
\end{theorem}

\noindent Here $\exp_2$ is the double iterated exponential function (i.e., $\exp_2=\exp\circ\exp$). The logarithm is a multivalued-function. Here, we define, as usual, $\log\zeta$ as the principal branch of the logarithm on the real line segment $(1,\infty)$ and elsewhere by continuous variation along line segments (as in \cite[p.\,210]{titchmarsh1986}).

We also provide a version of this result for the zeta-function in place of its logarithm (see Theorem \ref{Thm 3} below). 
Applying the first theorem to a partial sum of the Taylor expansion of the zeta-function and an admissible target function, we shall derive 

\begin{theorem}\label{Thm 2}
Let $s_0=\sigma_0+it_0$ be a fixed complex number with $\sigma_0\in (\frac 12,1)$ and for $r>0$ let $g\,:\,\mathcal K=\{s\in \mathbb C:|s-s_0|\le r\}\to \mathbb C$ be a continuous non-vanishing function which is analytic in the interior. Let $\delta_0\in(0,1)$, $\varepsilon\in (0,\min\{1,|g(s_0)|\})$. Assume that $N=N(\delta_0,\varepsilon)$ is a positive integer for which 
\[
\Big(\max_{|s-s_0|=r}|g(s)|\Big)\frac{\delta_0^N}{1-\delta_0}<\frac{\varepsilon}3.
\]
Then, there exists $\tau\in [T-t_0,T+H-t_0]$ for which 
$T^{\frac{27}{82}}\le H\le T$ and
\[
\max_{|s-s_0|\le\delta r}\big|\zeta(s+i\tau)-g(s)\big|<\varepsilon
\]
for any $\delta\in [0,\delta_0]$ satisfying
\[\Big(\max_{|s-s_0|=r}|\zeta(s+i\tau)|\Big)\frac{\delta^{N}}{1-\delta}<\frac \varepsilon 3\]
provided that
\[
T\ge \max\bigg\{\exp_2\bigg(C_{2}(N,\sigma_0)\Big(B(N,g,s_0,\delta_0,r,\varepsilon)\Big)^{{\frac{8}{1-\sigma_0}}+{\frac{8}{\sigma_0-{\frac 1 2}}}}\bigg),r\bigg\},
\]
Here $C_2(N,\sigma_0)$ is a positive, effectively computable constant depending on $N,\sigma_0$, and
\[B(N,g,s_0,\delta_0,r,\varepsilon)\coloneqq|\log g(s_0)|+\frac{(1+|g(s_0)|)\exp(\delta_0r)}{\varepsilon}\Big(\frac{\|\boldsymbol{G}\|}{|g(s_0)|}\Big)^{(N-1)^2},\]
where $\|\boldsymbol{G}\|\coloneqq\sum_{0\le k\le N-1}\big|\frac{d^{k}}{d s^k}g(s)|_{s=s_0}\big|$.
\end{theorem}

The proofs of these results will be given in the following two sections. Our reasoning follows Voronin's proof \cite{voronin1992,voronin1989}; note that these sources are not easy to read due to many typos and inaccuracies. In the final section, we consider what can be proved under the assumption of the Riemann hypothesis and conclude with a few final remarks.

\section{The Proof of Theorem \ref{Thm 1}}
Let $\Omega$ be the set of all infinite vectors $\boldsymbol\theta=(\theta_2,\theta_3,\dots,\theta_p,\dots)$, where each $\theta_p$ is a real number and the index $p$ runs through the set of all primes $\mathbb P$ in ascending order. 
We begin with an auxiliary result and introduce the following notation: let $\widehat Q$ denote the set of prime numbers $p\leq Q$ and
\begin{equation}\label{eq: 11}
\zeta_{\mathcal P}(s,\boldsymbol{\theta})\coloneqq\prod_{p\in {\mathcal P}}\Big(1-\frac{\exp(-2\pi i\theta_p)}{p^s}\Big)^{-1}
\end{equation}
for a set of primes ${\mathcal P}$. Our auxiliary result can be stated as follows:

\begin{prop}\label{thm: 1.1}
Under the assumptions from Theorem \ref{Thm 1}, there exists an effective computable constant $c_1(\sigma_0,N)>0$ (in terms of functions of $\sigma_0$ and $N$) such that if 
\begin{equation}\label{eq: thm: 1}
Q\ge c_1(\sigma_0,N)\Big(\|\boldsymbol{a}\|+\frac{1}{\varepsilon}\Big)^{\frac{8}{1-\sigma_0}+\frac 8{\sigma_0-\frac 12}},
\end{equation}
then there exists $\boldsymbol{\theta}_0=(\theta_2^{(0)},\theta_3^{(0)},\dots,\theta_p^{(0)},\dots)\in \Omega$ such that
\[
\bigg|\frac{d^k}{ds^k}\log \zeta_{\widehat Q}(s,\boldsymbol{\theta}_0)\Big|_{s=\sigma_0}-a_k\bigg|<\varepsilon
\]
for $k=0,1,\dots,N-1$.
\end{prop}

\begin{proof} 
Define $\boldsymbol{\theta}_1= (\theta_2^{(1)},\theta_3^{(1)}\dots,\theta_p^{(1)},\dots)\coloneqq(0,{\frac 12},0,{\frac 12},\ldots)$. Observe that for each $k=0,1,\dots,N-1$, the series
\[
\sum_{p}\frac{d^k}{ds^k}\log\Big(1-\frac{\exp(-2\pi i\theta_p^{(1)})}{p^s}\Big)^{-1}\Big|_{s=\sigma_0}
\]
converges to a complex number, which we will denote by $\gamma_k$.
For a non-empty set of primes ${\mathcal P}$, define 
\[
\varphi_{\mathcal P}(s,\boldsymbol\theta)\coloneqq\sum_{p\in {\mathcal P}}\frac{\exp(-2\pi i\theta_p)}{p^s}.
\]
Let $U_0>200^N$ and put $U_j\coloneqq U_02^j$ for $j<N$. Furthermore, let $V<U_0$ and $\mathcal M\coloneqq \bigcup_{j=0}^{N-1}M_j$, where
\[
M_j\coloneqq\big\{p\in \mathbb P: U_j\le p < U_j+V\big\}
\]
for $j=0,1,\dots,N-1$.
Then, for each $j,k=0,1,\dots,N-1$, we have
\begin{equation*}
\frac{\partial^k}{\partial s^k}\varphi_{M_j}(s,\boldsymbol{\theta})=(-\log U_j)^k\varphi_{M_j}(s,\boldsymbol{\theta})+R_{j,k},
\end{equation*}
where
\begin{equation}\label{eq: 02}
R_{j,k}\ll_N \frac{V^2(\log U_0)^{N-1}}{U_0^{\sigma_0+1}} 
\end{equation}
follows from the mean-value theorem.
Here and in the sequel $\ll_N$ means that the implicit constant depends only on $N$.
We consider the following system of linear equations in the unknowns $z_j$:
\[
\sum_{0\le j< N}(-\log U_j)^kz_j=a_k-\gamma_k
\]
for $k=0,1,\dots,N-1$. Observe that the coefficient matrix of this system is a Vandermonde matrix with a non-vanishing determinant. Applying Cramer's rule, there exists a unique solution $\boldsymbol{z}=(z_0,z_1,\dots,z_{N-1})$ satisfying
\begin{equation}\label{eq: 03}
\|\boldsymbol{z}\|\ll_N (\log U_0)^{N-1}\|\boldsymbol{a}-\boldsymbol{\gamma}\|,
\end{equation}
where $\|\boldsymbol z\|\coloneqq\sum_{k=0}^{N-1}|z_k|$ and $\|\boldsymbol a-\boldsymbol{\gamma}\|\coloneqq\sum_{k=0}^{N-1}|a_k-\gamma_k|$. A precise detail can be found in \cite[Lemma 2.1]{endo2023}.

Next, we determine how $U_0$ must be chosen in order that the equations
\begin{equation}\label{eq: 04}
\varphi_{M_j}(\sigma_0,\boldsymbol \theta)=z_j
\end{equation}
for $j=0,1,\dots,N-1$ can be solved for $\boldsymbol \theta$.
If $M_j$ contains at least three primes, then the set of values of $\varphi_{M_j}(\sigma_0,\boldsymbol{\theta})$ as $\boldsymbol{\theta}$ varies is a disk of radius 
\begin{equation}\label{eq: 05}
\sum_{p\in M_j}\frac 1{p^{\sigma_0}}.  
\end{equation} 
We need a lower bound for this quantity. This follows from a geometric argument; see \cite[\S 11.5]{titchmarsh1986}.

For this purpose, we apply a prime number theorem for short intervals due to Martin Huxley. If $\pi(x)$ is counting the prime numbers $p\leq x$, then 
\begin{equation}\label{eq: 06}
\pi(x+h)-\pi(x)\sim \frac{h}{\log x}
\end{equation}
for any $h\geq x^\theta$ with $\theta>{\frac 7{12}}$ as $x\to \infty$; a proof can be found in \cite[Chapter 28]{huxley}.

Using this in combination with partial summation implies that if $U_0$ and $V$ satisfy
\begin{equation}\label{eq: 07}
U_0\gg_{N,\sigma_0} \big(\|\boldsymbol{a}-\boldsymbol{\gamma}\|+1\big)^{\frac{8}{1-\sigma_0}}
\quad \text{and}\quad V\coloneqq U_0^{\frac{1+3\sigma_0}4},
\end{equation}
then
\begin{equation}\label{eq: 08}
\sum_{p\in M_j}\frac 1{p^{\sigma_0}}\gg_N \sum_{p\in M_j}\frac 1{U_0^{\sigma_0}}
\gg_N \frac{V}{U_0^{\sigma_0}\log U_0} \gg_{N,\sigma_0} (\log U_0)^{N-1}(\|\boldsymbol{a}-\boldsymbol{\gamma}\|+1).
\end{equation}
Note that $\frac{1+3\sigma_0}4>\frac 7{12}$, hence Huxley's prime number theorem is applicable.

In view of \eqref{eq: 03}, \eqref{eq: 05} and \eqref{eq: 08} it follows that the system of linear equations \eqref{eq: 04} can be solved for $\boldsymbol \theta=\boldsymbol \theta_2\coloneqq(\theta_2^{(2)},\theta_3^{(2)},\dots,\theta_p^{(2)},\dots)$ provided that \eqref{eq: 07} holds.
In addition, if
\begin{equation}\label{eq: 09}
U_0\gg_{N,\sigma_0} \Big(\frac{1}{\varepsilon}\Big)^{\frac 4{1-\sigma_0}},
\end{equation}
then \eqref{eq: 02} implies that
\[
\sum_{0\le j<N}|R_{j,k}| \ll_{N} U_0^{-\frac{1-\sigma_0}{2}}(\log U_0)^{N-1} \ll_{N,\sigma_0} \varepsilon.
\]
Thus, the condition that $U_0$ satisfies \eqref{eq: 07} and \eqref{eq: 09} implies 
\begin{equation*}
\bigg|\sum_{0\le j< N}\frac{d^k}{ds^k}\varphi_{M_j}(s,\boldsymbol \theta_2)\Big|_{s=\sigma_0}-(a_k-\gamma_k)\bigg|
=\bigg|\sum_{0\le j< N}R_{j,k}\bigg|< \frac \varepsilon{4}.
\end{equation*}
Since $\gamma_k$ is an absolute constant, it suffices to assume 
\begin{equation}\label{eq: 10}
U_0\gg_{N,\sigma_0}  \Big(\|\boldsymbol{a}\|+\frac 1{\varepsilon}\Big)^{\frac 8{1-\sigma_0}}
\end{equation}
instead of \eqref{eq: 07} and \eqref{eq: 09}.

Obviously, the logarithm of $\zeta_{\mathcal P}$ is related to $\varphi_{\mathcal P}$.
It follows from the Taylor expansion and partial summation that
\begin{equation}\label{eq: 12}
\bigg|\frac{\partial^k}{\partial s^k}\log \Big(1-\frac{\exp(-2\pi i\theta_p)}{p^s}\Big)^{-1}\Big|_{s=\sigma_0}-\frac{\partial^k}{\partial s^k}\frac{\exp(-2\pi i\theta_p)}{p^s}\Big|_{s=\sigma_0}\bigg|\ll_N \frac{(\log p)^k}{p^{2\sigma_0}}.
\end{equation}
This implies that for any $j=0,1,\dots,N-1$,
\begin{align*}
\bigg|\frac{\partial^k}{\partial s^k}\log \zeta_{M_j}(s,\boldsymbol{\theta})\Big|_{s=\sigma_0}-\frac{\partial^k}{\partial s^k}\varphi_{M_j}(s,\boldsymbol{\theta})&\Big|_{s=\sigma_0}\bigg| \\
&\ll_{N} \sum_{p\in M_j}\frac{(\log p)^k}{p^{2\sigma_0}} \ll_N U_0^{-\frac 38}(\log U_0)^{N-1}.
\end{align*}
Hence, by possibly increasing the implicit constant, the system of inequalities
\begin{equation}\label{eq: 13}
\bigg|\sum_{0\le j< N}\frac{d^k}{ds^k}\log{\zeta_{M_j}}(s,\boldsymbol{\theta}_2)\Big|_{s=\sigma_0}-(a_k-\gamma_k)\bigg|<\frac{\varepsilon}2
\end{equation}
is solvable provided \eqref{eq: 10} holds.

We claim that for $Q>U_02^N$,
\begin{equation}\label{eq: 14}
\sum_{p\in \widehat Q\setminus \mathcal M}\frac{d^k}{ds^k}\log\Big(1-\frac{\exp(-2\pi i\theta_p^{(1)})}{p^{s}}\Big)^{-1}\Big|_{s=\sigma_0}=\gamma_k +O_N\Big( U_{0}^{1-2\sigma_0}(\log U_0)^N\Big).
\end{equation}
To see this, it suffices to show that
\[\bigg\{\sum_{p>Q}+\sum_{p\in \mathcal M}\bigg\}\frac{d^k}{ds^k}\log\Big(1-\frac{\exp(-2\pi i\theta_p^{(1)})}{p^{s}}\Big)^{-1}\Big|_{s=\sigma_0}\ll_N U_{0}^{1-2\sigma_0}(\log U_0)^N.\]
By \eqref{eq: 12}, the prime number theorem, and partial summation, we have
\begin{align*}
&\bigg\{\sum_{p>Q}+\sum_{p\in \mathcal M}\bigg\}
\bigg|\frac{d^k}{ds^k}\log \Big(1-\frac{\exp(-2\pi i\theta_p^{(1)})}{p^s}\Big)^{-1}\Big|_{s=\sigma_0}-\frac{d^k}{ds^k}\frac{\exp(-2\pi i\theta_p^{(1)})}{p^s}\Big|_{s=\sigma_0}\bigg|\\
& \ll \sum_{p>Q}\frac{(\log p)^k}{p^{2\sigma_0}}+\sum_{p\in \mathcal M}\frac{(\log p)^k}{p^{2\sigma_0}} \ll Q^{1-2\sigma_0}(\log Q)^{k-1}+U_0^{1-2\sigma_0}(\log U_0)^{N-1}.
\end{align*}
In addition, it is not difficult to see that
\[
\sum_{p>Q}\frac{d^k}{ds^k}\frac{\exp(-2\pi i\theta_p^{(1)})}{p^s}\Big|_{s=\sigma_0}=\sum_{p>Q}\frac{(-\log p)^k\exp(-2\pi i\theta_p^{(1)})}{p^{\sigma_0}}\ll \frac{(\log Q)^k}{Q^{\sigma_0}}
\]
and 
\[
\sum_{p\in \mathcal M}\frac{d^k}{ds^k}\frac{\exp(-2\pi i\theta_p^{(1)})}{p^s}\Big|_{s=\sigma_0}=\sum_{p\in \mathcal M}\frac{(-\log p)^k\exp(-2\pi i\theta_p^{(1)})}{p^{\sigma_0}}\ll_N \frac{(\log U_0)^k}{U_0^{\sigma_0}}.
\]
Hence, we obtain \eqref{eq: 14} and then we have
\begin{equation}\label{eq: 15}
\bigg|\sum_{p\in \widehat Q\setminus \mathcal M}\frac{d^k}{ds^k}\log\Big(1-\frac{\exp(-2\pi i\theta_p^{(1)})}{p^{s}}\Big)^{-1}\Big|_{s=\sigma_0}
-\gamma_k\bigg|<\frac{\varepsilon}2
\end{equation}
provided that
\begin{equation}\label{eq: 16}
U_0\gg_{N,\sigma_0} \Big(\frac 1\varepsilon\Big)^{\frac{1}{\sigma_0-\frac 12}}.
\end{equation}

Consequently, we choose $\boldsymbol{\theta}_0\coloneqq(\theta_2^{(0)},\theta_3^{(0)},\dots,\theta_p^{(0)},\dots)$ defined by
\[\theta_p^{(0)}\coloneqq\begin{cases}
\theta_p^{(1)} & \text{for } p\in \widehat Q\setminus \mathcal M,\\
\theta_p^{(2)} & \text{for } p\in \mathcal M.
\end{cases}\]
Then, by \eqref{eq: 13} and \eqref{eq: 15}, we obtain, for $k=0,\dots,N-1$, that
\[
\bigg|\frac{d^k}{ds^k}\log \zeta_{\widehat Q}(s,\boldsymbol{\theta}_0)\Big|_{s=\sigma_0}-a_k\bigg|<\varepsilon
\]
provided that \eqref{eq: 10} and \eqref{eq: 16} hold. Thus, we prove the proposition.
\end{proof}

Following the original proof of Voronin, the next step involves some Fourier ana\-ly\-sis. To do this, let $\lambda$ be a real-valued, infinitely differentiable function satisfying
\[
0\leq \lambda(x)\leq 1,\quad \supp \lambda\subseteq [-1,1], \qquad \text{and}\qquad \int_{-\infty}^\infty\lambda(x) dx=1.\]
For $\delta=\delta(Q)\in (0,\frac 12)$ and $\theta_p\in (-\frac 12,\frac 12]$, define
\begin{equation*}
L_Q(\boldsymbol{\theta})\coloneqq\prod_{p\in \widehat{Q}}\lambda_\delta(\theta_p),
\end{equation*}
where $\lambda_\delta$ is given by
\[
\lambda_\delta(\theta)\coloneqq\frac 1\delta\lambda\Big(\frac{\theta}{\delta}\Big).
\]
Recall that $\widehat Q$ is the set of all primes not exceeding $Q$ and $\boldsymbol{\theta}=(\dots,\theta_p,\dots)\in \Omega$.
Now we extend $L_Q(\boldsymbol{\theta})$ to all of $\Omega$ by periodicity with period $1$ in each variable~$\theta_p$. Then, the function $\frac 1\delta \lambda(\frac \theta\delta)$, also extended with period $1$ to $\mathbb R$, can be represented as a Fourier series
\[
\frac 1\delta\lambda\Big(\frac \theta\delta\Big)=\sum_{n=-\infty}^{\infty}\alpha_n\exp(2\pi in\theta),
\]
where $\alpha_0= 1$ and $\alpha_n \ll n^{-2}\delta^{-3}$ for every non-zero integer $n$. The last relation follows from integration by parts twice and the implicit constant depends only on our choice of $\lambda$. It thus follows that $L_Q(\boldsymbol{\theta})$ has a Fourier expansion,
\[
L_Q(\boldsymbol{\theta})=\sum_{\boldsymbol{n}}\beta_{\boldsymbol{n}}\exp(2\pi  i \langle \boldsymbol{n},\boldsymbol{\theta}\rangle),
\]
where $\boldsymbol{n}$ is extended from $(n_p)_{p\in \widehat Q}\in \mathbb Z^{\pi(Q)}$ by setting $n_p\coloneqq 0$ for all primes $p>Q$ and
\[
\beta_{\boldsymbol{n}}\coloneqq \prod_{p\in \widehat Q}\alpha_{n_p}.
\]
Hence, for every positive integer $M$,
\begin{align*}
L_Q(\boldsymbol \theta)=\sum_{\max|n_p|\le M}\beta_{\boldsymbol n}\exp\big(2\pi i\langle \boldsymbol n,&\boldsymbol \theta \rangle
\big)+\\
&+O\bigg(\pi(Q)\Big(\sum_{|n_p|>M}|\alpha_{n_p}|\Big)\Big(\sum_{n\in \mathbb Z}|\alpha_n|\Big)^{\pi(Q)-1}\bigg).
\end{align*}
Obviously, 
\[
\sum_{|n|>M}|\alpha_n|\ll \frac 1{\delta^3 M} \quad
\text{and}\quad \sum_{n}|\alpha_n|\ll \frac 1 {\delta^{3}},
\]
as well as
\[
\beta_{\boldsymbol 0}=1\quad \text{and}\quad
\beta_{\boldsymbol{n}}\ll \prod_{p\le Q}\min\{1,\delta^{-3}n_p^{-2}\}.
\]
Hence, we have
\begin{equation}\label{eq: 1.2}
L_Q(\boldsymbol \theta)
=1+\sum_{\begin{subarray}{c}
\max{\vert n_p\vert }\le M\\ \boldsymbol{n}\ne \boldsymbol{0}\end{subarray}}\beta_{\boldsymbol n}\exp(2\pi i\langle \boldsymbol n,\boldsymbol \theta\rangle)+O\bigg(\frac{Q\exp(3\pi(Q)\log {\frac 1\delta})}{M\log Q}\bigg)
\end{equation}
and
\begin{equation}\label{eq: 1.3}
\sum_{\boldsymbol{n}}\vert\beta_{\boldsymbol n}\vert\ll \bigg(\sum_{n\in \mathbb Z}\min\{1,\delta^{-3}n^{-2}\}\bigg)^{\pi(Q)}\ll \exp(3Q).
\end{equation}

Our next result deals with the integral of $L_Q$ with respect to a uniformly distributed curve. For this purpose, define the curve $\boldsymbol{\gamma}:\mathbb R \to \Omega$ by
\begin{equation}\label{eq: 1.4}
\boldsymbol \gamma(t)\coloneqq\left(\frac{\log 2}{2\pi}t,\frac{\log 3}{2\pi}t,\dots,\frac{\log p}{2\pi}t,\dots\right).
\end{equation}

\begin{prop}\label{thm: 2}
According to the above notation, let $\delta=Q^{-1}$, $T^{\nu}\le H\le T$, where $0<\nu<1$, and $\boldsymbol \theta\in \Omega$. Then, for $M\gg Q^2\exp(3Q)$  and $T\gg \big(Q\exp((3+M)Q)\big)^{\frac 1\nu}$,
\[
\bigg|\frac 1{H}\int_{T}^{T+H}L_Q\big(\boldsymbol \gamma(t)-\boldsymbol \theta\big) dt-1\bigg|<\frac 1Q.
\]
\end{prop}

\begin{proof}
We begin with substituting $\delta=Q^{-1}$ in the Fourier expansion \eqref{eq: 1.2} and using the prime number theorem to obtain
\begin{equation}\label{eq: 1.5}
L_Q(\boldsymbol \gamma(t)-\boldsymbol \theta)
=1+\sum_{\begin{subarray}{c}
\max{|n_p|}\le M\\
\boldsymbol{n} \neq \boldsymbol{0}\end{subarray}}\beta_{\boldsymbol n}\exp\big(2\pi i\langle \boldsymbol n,\boldsymbol \gamma(t)-\boldsymbol\theta\rangle\big)+O\Big(\frac{Q\exp(3Q)}{M\log Q}\Big).
\end{equation}

To continue, recall that 
\begin{equation}\label{eq: 1.6}
\sum_{0<m<n<T}\frac 1{m^\sigma n^\sigma \log nm^{-1}}\ll T^{2-2\sigma}\log T
\end{equation}
for $\frac 12\le \sigma\le 1$, and uniformly for $\frac 12\le \sigma\le \sigma_0<1$. Moreover, 
\begin{equation}\label{eq: 1.7}
\Big|\log \frac mn\Big|>\frac 1{\max(m,n)}
\end{equation}
for distinct positive integers $m,n$ (see \cite[\S 7.2, p.\,139 and \S 8.8, p.\,193]{titchmarsh1986}).

Then, for each $\boldsymbol{n}$ such that $\max|n_p|\le M$, we have
\begin{align}
\int_T^{T+H}\exp\big( 2\pi  i \langle \boldsymbol n,\boldsymbol{\gamma}(t)-\boldsymbol \theta\rangle\big)  dt
&=\int_T^{T+H}\exp\bigg( 2\pi  i\Big(\sum_{p\le Q}\frac{n_p(\log p)t }{2\pi}+n_p\theta_p\Big)\bigg) dt \notag\\
&\ll \bigg|\int_T^{T+H}\exp\Big( it\sum_{p\le Q}n_p\log p\Big) dt\bigg|\notag\\
& \ll \bigg(\sum_{p\le Q}n_p\log p\bigg)^{-1}\ll \exp(MQ)\label{eq: 1.8}
\end{align}
by \eqref{eq: 1.7} and the prime number theorem.
Hence, by \eqref{eq: 1.3} and \eqref{eq: 1.5}, we get that
\begin{align*}
\frac 1{H}\int_{T}^{T+H}L_Q\big(\boldsymbol \gamma(t)-\boldsymbol \theta\big) dt-1
&\ll \frac{\exp(MQ)}H
\bigg(\sum_{\begin{subarray}{c}
\max{|n_p|}\le M\\
\boldsymbol{n}\ne \boldsymbol{0}\end{subarray}}|\beta_{\boldsymbol{n}}|\bigg)+\frac{Q\exp(3Q)}{M\log Q}\\
& \ll \frac{\exp((3+M)Q)}H+\frac{Q\exp(3Q)}{M\log Q}.
\end{align*}
For $M\gg Q^2\exp(3Q)$ and $T\gg \big(Q\exp\big((3+M)Q\big)\big)^{\frac 1\nu}$, we deduce the desired inequality.
\end{proof}

For the next step, assume that there is $\sigma^*\in (\frac 12,\sigma_0)$ such that $\omega(\sigma^*)<1$, where
\begin{equation}\label{eq: 3.N}
N(\alpha;T,H)\ll H^{\omega(\alpha)}(\log H)^{\eta(\alpha)}  
\end{equation}
uniformly in $\alpha$ and $H$, for all $H$ satisfying $T^\nu\le H\le T$. We consider the integral
\begin{align*}
&I\coloneqq \\
&\int_{D_{T,H}}\sum_{0\le k<N}L_Q\big(\boldsymbol{\gamma}(t)-\boldsymbol \theta\big)\bigg|\frac{d^k}{ds^k}\log \zeta(s)\Big|_{s=\sigma_0+ it}-\frac{d^k}{ds^k}(\log \zeta_{\widehat Q}(s+ it,\boldsymbol 0))\Big|_{s=\sigma_0}\bigg|^2 dt,
\end{align*}
where $\boldsymbol{\theta}\in \Omega$ and $D_{T,H}\subseteq [T,T+H]$ is defined as follows: Around every hypo\-the\-ti\-cal exceptional zero $\rho=\beta+ i\gamma$ of $\zeta(s)$ in $\sigma>\sigma^*$, where we write $s=\sigma+it$, as usual, we draw a rectangle
\[
P_\rho^{(h)}\coloneqq \big\{s=\sigma+it:\sigma^*<\sigma<2,|t-\gamma|\le h\big\},
\]
where $h\in [10,H)$ is a parameter. Then, define
\begin{equation*}
D_{T,H}\coloneqq\Big\{t\in [T,T+H]:\sigma_0+ it\notin\bigcup_{\rho}P_{\rho}^{(h)}\Big\}.  
\end{equation*}

Our aim is to bound $I$ from above, for suitable $Q$ and $T$, by
\[I\ll_{N,\sigma,\nu}\varepsilon^2\int_{D_{T,H}}L_Q\big(\boldsymbol{\gamma}(t)-\boldsymbol{\theta}\big)dt.\]
For this purpose, we use an explicit formula due to Atle Selberg \cite{selbe}, that is 
\begin{align}
\frac{\zeta'}{\zeta}(s)&=-\sum_{n\le x^2}\frac{\Lambda_x(n)}{n^s}+\frac{x^{2(1-s)}-x^{1-s}}{(1-s)^2\log x}+\notag\\
&\qquad +\frac 1{\log x}\sum_{q\ge 1}\frac{x^{-2q-s}-x^{-2(2q+s)}}{(2q+s)^2}+\frac 1{\log x}\sum_{\rho}\frac{x^{\rho-s}-x^{2(\rho-s)}}{(s-\rho)^2},\label{eq: 2.2}
\end{align}
where 
\[
\Lambda_x(n)\coloneqq\begin{cases}
\Lambda(n) & \text{if $1\le n<x$},\\
\Lambda (n){\frac{\log(x^2/n)}{\log x}} & \text{if $x\le n\le x^2$},\\
0 & \text{if $n>x^2$},
\end{cases}
\]
and $\Lambda$ is the von Mangoldt $\Lambda$-function (see \cite{titchmarsh1986}). 
Integrating this identity, we obtain
\begin{align}
\log \zeta(s)&=\sum_{n\le x^2}\frac{\Lambda_x(n)}{n^s\log n}+\sum_{n>x}\frac{\Lambda(n)-\Lambda_x(n)}{n^{s+10}\log n}+\notag\\
&\qquad -\frac 1{\log x}F(s,1)+\frac 1{\log x}\sum_\rho F(s,\rho)+\frac 1{\log x}\sum_{q\ge 1}F(s,-2q),\label{eq: 2.3}
\end{align}
where
\[
F(s,z)\coloneqq\int_{s+10}^s\frac{x^{z-w}-x^{2(z-w)}}{(w-z)^2}dw.
\]
{When $k=1,2,\dots,N-1$, we use \eqref{eq: 2.2} to estimate $I$, and when $k=0$, we use \eqref{eq: 2.3}.

We choose $Q,x$ and $T$ such that $Q<x\ll H^{\alpha}$ for some $\alpha\in (0,1)$. Defining
\[
I_k\coloneqq\int_{D_{T,H}}L_Q\big(\boldsymbol{\gamma}(t)-\boldsymbol \theta\big)\bigg|\frac{d^k}{d s^k}\log \zeta(s)\Big|_{s=\sigma_0+ it}-\frac{d^k}{ds^k}\log \zeta_{\widehat Q}(s+ it,\boldsymbol 0)\Big|_{s=\sigma_0}\bigg|^2 dt,
\]
we have
\[
I=\sum_{k=0}^{N-1}I_k.
\]
By \eqref{eq: 2.2}, we conclude that
\[I_k\ll A_k+B_k+C_k+D_k+E_k,\]
where
\begin{align*}
A_k &\coloneqq\int_{D_{T,H}}L_Q\big(\boldsymbol{\gamma}(t)-\boldsymbol \theta\big)\\
&\qquad\qquad\qquad\qquad\bigg|
\frac{d^{k-1}}{d s^{k-1}}\sum_{m\le x^2}\frac{\Lambda_x(m)}{m^s}\Big|_{s=\sigma_0+ it}-\frac{d^{k-1}}{d s^{k-1}}\sum_{p^\ell\le Q}\frac{\log p}{p^{\ell (s+it)}}\Big|_{s=\sigma_0}\bigg|^2 dt,\\
B_k &\coloneqq\int_{D_{T,H}}L_Q\big(\boldsymbol{\gamma}(t)-\boldsymbol \theta\big)\bigg|
\frac{d^{k-1}}{d s^{k-1}}\sum_{\begin{subarray}{c}
p^\ell >Q \\ p\le Q\end{subarray}}\frac{\log p}{p^{\ell (s+it)}}\Big|_{s=\sigma_0}\bigg|^2 dt,\\
C_k &\coloneqq\frac 1{(\log x)^2}\int_{D_{T,H}}L_Q\big(\boldsymbol{\gamma}(t)-\boldsymbol \theta\big)\bigg|\frac{d^{k-1}}{d s^{k-1}} \sum_{q\ge 1}\frac{x^{-2q-s}-x^{-2(2q+s)}}{(2q+s)^2}\Big|_{s=\sigma_0+ it}\bigg|^2 dt,\\
D_k &\coloneqq\frac 1{(\log x)^2}\int_{D_{T,H}}L_Q(\boldsymbol{\gamma}(t)-\boldsymbol \theta)\bigg|\frac{d^{k-1}}{d s^{k-1}} \frac{x^{2(1-s)}-x^{1-s}}{(1-s)^2}\Big|_{s=\sigma_0+ it}\bigg|^2  dt,\\
E_k &\coloneqq\frac 1{(\log x)^2}\int_{D_{T,H}}L_Q(\boldsymbol{\gamma}(t)-\boldsymbol \theta)\bigg|\frac{d^{k-1}}{d s^{k-1}}\sum_{\rho}\frac{x^{\rho-s}-x^{2(\rho-s)}}{(s-\rho)^2}\Big|_{s=\sigma_0+ it}\bigg|^2 dt.
\end{align*}

We assume the assumption in Proposition \ref{thm: 2}: let $\delta=Q^{-1}$, $M\gg Q^2\exp(3Q)$  and $T\gg (Q\exp((3+M)Q))^{\frac 1\nu}$.
From the definition of $\Lambda_x$, we have
\begin{align*}
\frac{d^{k-1}}{d s^{k-1}}\sum_{m\le x^2}\frac{\Lambda_x(m)}{m^s}\Big|_{s=\sigma_0+ it}-\frac{d^{k-1}}{d s^{k-1}}\sum_{p^\ell\le Q}\frac{\log p}{p^{\ell (s+it)}}\Big|_{s=\sigma_0}\qquad\qquad\qquad\qquad\\
=\sum_{Q<p\le x^2}\frac{a_{p}}{p^{{\sigma_0+ it}}}+\sum_{\begin{subarray}{c}
Q<p^\ell\le x^2\\ \ell\ge 2\end{subarray}}\frac{a_{p,\ell}}{p^{\ell (\sigma_0+ it)}}
\end{align*}
with implicitly defined coefficients $a_p$ and $a_{p,\ell}$. We define $A_{k,1}$ and $A_{k,2}$ according to this splitting such that
\[
A_k\ll A_{k,1}+A_{k,2}.
\]

To estimate $A_{k,1}$, we use again \eqref{eq: 1.6} and \eqref{eq: 1.7}. These give, in combination with partial summation and the prime number theorem,
\begin{align*}
\int_T^{T+H}&\bigg|\sum_{Q<p\le x^2}\frac{a_p}{p^{\sigma_0+ it}}\bigg|^2 dt\\
&=\sum_{Q<p_1,p_2\le x^2}\frac{a_{p_1}\overline a_{p_2}}{(p_1p_2)^{\sigma_0}}\bigg|\int_T^{T+H}\exp\big(- it\log (p_1p_2^{-1})\big) dt\bigg|\\
&= H\sum_{Q<p\le x^2}\frac{|a_p|^2}{p^{2\sigma_0}}+O\bigg(\sum_{\begin{subarray}{c}
Q<p_1,p_2\le x^2\\
p_1\ne p_2\end{subarray}}\frac{|a_{p_1}||a_{p_2}|}{(p_1p_2)^{\sigma_0}\log(p_1p_2^{-1})}\bigg)\\
&\ll_N HQ^{1-2\sigma_0}(\log Q)^{2k-1}+x^{2(2-2\sigma_0)}(\log x)^{2k+1}.
\end{align*}
In view of \eqref{eq: 1.8} we note that
\begin{align*}
\bigg|\int_T^{T+H}\exp\big( 2\pi  i \langle \boldsymbol n,\boldsymbol{\gamma}(t)-\boldsymbol \theta\rangle\big)\Big|\sum_{Q<p\le x^2}\frac{a_p}{p^{\sigma_0+ it}}\Big|^2 dt\bigg|
\ll x^4(\log x)^{2k}Q^{M\pi(Q)-2\sigma_0}.
\end{align*}
By Fourier expansion \eqref{eq: 1.5} together with \eqref{eq: 1.3}, the prime number theorem and $M\gg Q^2\exp(3Q)$, we have
\begin{align*}
A_{k,1}&\coloneqq \int_T^{T+H}L_Q\big(\boldsymbol{\gamma}(t)-\boldsymbol \theta\big)\bigg|\sum_{Q<p\le x^2}\frac{a_p}{p^{\sigma_0+it}}\bigg|^2 dt\\
&\ll_N HQ^{1-2\sigma_0}(\log Q)^{2k}+x^{2(2-2\sigma_0)}(\log x)^{2k+1}\\
&\qquad\qquad\qquad\qquad\qquad\qquad\qquad\qquad+ x^4(\log x)^{2k}Q^{M\pi(Q)}\sum_{\begin{subarray}{c}
\max{|n_p|}\le M\\ \boldsymbol n\ne \boldsymbol 0\end{subarray}} |\beta_{\boldsymbol{n}}|\\
&\ll_N HQ^{1-2\sigma_0}(\log Q)^{2k}+x^4(\log x)^{2k+1}\exp\big((3+M)Q\big).
\end{align*}
To estimate $A_{k,2}$, we apply partial summation together with the prime number theorem, and then get
\begin{align*}
\sum_{\begin{subarray}{c}
Q<p^\ell<x^2\\ \ell\ge 2
\end{subarray}}\frac{a_{p,\ell}}{p^{\ell ({\sigma_0+ it})}}
&\ll_N \sum_{p\le \sqrt{Q}}\sum_{\ell>\frac{\log Q}{\log p}}\frac{(\ell \log p)^k}{p^{\ell \sigma_0}}+ \sum_{p>\sqrt{Q}}\sum_{\ell \ge 2}\frac{(\ell \log p)^k}{p^{\ell \sigma_0}}\\
& \ll_N \sum_{p\le \sqrt Q} \frac{(\log p)^k}{Q^{\sigma_0}}+\sum_{p>\sqrt Q}\frac{(\log p)^k}{p^{2\sigma_0}} 
\ll_N Q^{\frac 12-\sigma_0}(\log Q)^{k-1}.
\end{align*}}
Hence,
\begin{equation}\label{eq: 2.A}
A_k
\ll_N HQ^{1-2\sigma_0}(\log Q)^{2k}+x^4(\log x)^{2k+1}\exp\big((3+M)Q\big).
\end{equation}

Next, we shall estimate $B_k$. By partial summation formula and the prime number theorem, we have
\begin{align*}
\bigg|\frac{d^{k-1}}{d s^{k-1}}\sum_{\begin{subarray}{c}
{p^\ell>Q}\\ p\le Q
\end{subarray}
}\frac{\log p}{p^{\ell (s+it)}}\Big|_{s=\sigma_0}\bigg|
&\le \sum_{p^\ell>Q}\frac{\ell^{k-1}(\log p)^k}{p^{\ell\sigma_0}}\\
&\le \sum_{p\le \sqrt{Q}}\sum_{\ell >\frac{\log Q}{\log p}}\frac{\ell^{k-1} (\log p)^k}{p^{\ell \sigma_0}}+ \sum_{p>\sqrt{Q}}\sum_{\ell \ge 2}\frac{\ell^{k-1}(\log p)^k}{p^{\ell \sigma_0}}\\
&\ll_N Q^{{\frac 12}-\sigma_0}(\log Q)^{k-1}.
\end{align*}
By a trivial estimation, we thus obtain
\begin{equation}\label{eq: 2.B}
B_k\ll_N Q^{1-2\sigma_0}(\log Q)^{2k-2}\int_{D_{T,H}}L_Q(\boldsymbol{\gamma}(t)-\boldsymbol \theta) dt.
\end{equation}

It remains to estimate integrals of the form
\[
\int_{D_{T,H}}L_Q\big(\boldsymbol \gamma(t)-\boldsymbol \theta\big)\bigg|\frac{d^{k-1}}{d s^{k-1}}f(s)\Big|_{s=\sigma_0+it}\bigg|^2 dt.
\]
By the Cauchy integral formula, the estimation
\begin{align*}
\max_{t\in D_{T,H}}\bigg|\frac{d^{k-1}}{d s^{k-1}}f(s)\Big|_{s=\sigma_0+it}\bigg|^2
&=\max_{t\in D_{T,H}}\bigg\vert\frac{(k-1)!}{2\pi i}\oint_{\widetilde{D}_{\sigma_0+it}}\frac{f(z)}{(z-\sigma_0-it)^k} dz\bigg\vert^2 \\
&\ll_{N,\sigma_0}\max_{s\in \widetilde{D}}|f(s)|^2
\end{align*}
holds for any holomorphic function $f(s)$ on
\begin{equation*}
\widetilde D\coloneqq\Big\{s=\sigma+it: |\sigma-\sigma_0|\le r_{\sigma_0}, \Big|t-T-\frac H2\Big|\le \frac H2+r_{\sigma_0}, \text{ and }s\notin \bigcup_{\rho}P_{\rho}^{(h-r_{\sigma_0})}\Big\}, 
\end{equation*}
where $r_{\sigma_0}=\frac 1{2}(\sigma^*+\sigma_0)$
and $\widetilde{D}_{\sigma_0+it}$ is a circle with center $\sigma_0+it$ and radius $r_{\sigma_0}$.
By a trivial estimation, we have
\begin{align*}
\int_{D_{T,H}}L_Q\big(\boldsymbol \gamma(t)-\boldsymbol \theta\big)\bigg|\frac{d^{k-1}}{d s^{k-1}}f(s)\Big|_{s=\sigma_0+it}&\bigg|^2 dt\\
&
\ll_{N,\sigma_0}\max_{s\in \widetilde{D}}|f(s)|^2\int_{D_{T,H}}L_Q\big(\boldsymbol{\gamma}(t)-\boldsymbol \theta\big) dt.
\end{align*}

Now we are ready to estimate $C_k,D_k$ and $E_k$. {Note that
\[\bigg|\sum_{q\ge 1}\frac{x^{-2q-s}-x^{-2(2q+s)}}{(2q+s)^2}\bigg|
\le \frac 1{(T-r_{\sigma_0})^2}\Big(\frac{x^{-2-\sigma}}{1-x^{-2}}+\frac{x^{-4-2\sigma}}{1-x^{-4}}\Big)
\ll x^{-{\frac 52}}T^{-2}\]
and
\[
\bigg|\frac{x^{2(1-s)}-x^{1-s}}{(1-s)^2}\bigg|\le \frac{x^{2-2\sigma}-x^{1-\sigma}}{(T-r_{\sigma_0})^2} \ll xT^{-2}
\]
for every $s\in \widetilde D$. This implies that
\begin{equation}\label{eq: 2.C}
C_k \ll_{N,\sigma_0}  \frac{1}{x^5T^4(\log x)^2}\int_{D_{T,H}}L_Q(\boldsymbol \gamma(t)-\boldsymbol \theta) dt    
\end{equation}
and
\begin{equation}\label{eq: 2.D}
D_k \ll_{N,\sigma_0} \frac{x^2}{T^4(\log x)^2}\int_{D_{T,H}}L_Q(\boldsymbol \gamma(t)-\boldsymbol \theta) dt.    
\end{equation}}
It remains to consider $E_k$.
For each $s\in \widetilde{D}$, we have
\begin{align*}
\sum_{\begin{subarray}{c}
\rho=\beta+i\gamma \\ \beta\ge \frac 1{2}(\sigma^*+\sigma_0)\end{subarray}}&|s-\rho|^{-2}\\
& \le \sum_{0\le n\le 2T+1}\sum_{\begin{subarray}{c}
\rho=\beta+i\gamma \\ 
h-1+n\le |t-\gamma|< h+n\end{subarray}}|s-\rho|^{-2} 
+\sum_{\begin{subarray}{c}
\rho=\beta+i\gamma\\
|t-\gamma|\ge h+2T\end{subarray}
}|s-\rho|^{-2} \\
& \ll \sum_{0\le n\le 2T+1}\frac{\log(t+h+n)}{(h-1+n)^2}+\sum_{\gamma>T}\gamma^{-2}\\
& \ll \log T\int_{h-1}^{2T+H}\frac{1}{u^2}du + \int_T^{\infty}\frac{1}{u^2}d N(u)
\ll \frac{\log T}h
\end{align*}
and
\begin{align*}
\sum_{\begin{subarray}{c}
\rho=\beta+i\gamma \\ \beta<\frac 1{2}(\sigma^*+\sigma_0)\end{subarray}}|s-\rho|^{-2}
& \le \sum_{1\le n\le 2T+1}\sum_{\begin{subarray}{c}
\rho=\beta+i\gamma\\
n-1\le |t-\gamma|< n\end{subarray}}|s-\rho|^{-2} +\sum_{\begin{subarray}{c}
\rho=\beta+i\gamma\\
|t-\gamma|\ge 2T
\end{subarray}}|s-\rho|^{-2} \\
& \ll \frac {\log T}{(\sigma-\sigma^*)^2}+\sum_{2\le n\le 2T+1}\frac{\log(t+n)}{(n-1)^2}+\sum_{\gamma>T}\gamma^{-2}\\
& \ll_{\sigma_0} \log T+\log T\int_{1}^{2T}\frac{1}{u^2}du + \int_T^{\infty}\frac{1}{u^2}d N(u)
\ll \log T,
\end{align*}
where $N(u)$ denotes the number of zeros $\rho=\beta+i\gamma$ of $\zeta(s)$ such that $0<\gamma \le u$.
Here the so-called Riemann-von Mangoldt formula gives an asymptotic formula for the number of zeros $N(u)$ (see \cite[Chapter 9]{titchmarsh1986}).

Hence, for each $s\in \widetilde{D}$, we have
\begin{align*}
\sum_{\rho}&\frac{x^{\rho-s}-x^{2(\rho-s)}}{(s-\rho)^{2}}\\
&=\sum_{\begin{subarray}{c}
\rho=\beta+i\gamma \\ \beta\ge \frac 1{2}(\sigma^*+\sigma_0)\end{subarray}}\frac{x^{\rho-s}-x^{2(\rho-s)}}{(s-\rho)^{2}}+\sum_{\begin{subarray}{c}
\rho=\beta+i\gamma \\\beta< \frac 1{2}(\sigma^*+\sigma_0)\end{subarray}}\frac{x^{\rho-s}-x^{2(\rho-s)}}{(s-\rho)^{2}}\\
& \ll x^{2(1-\sigma)}\sum_{\begin{subarray}{c}
\rho=\beta+i\gamma \\ \beta\ge \frac 1{2}(\sigma^*+\sigma_0)\end{subarray}}|s-\rho|^{-2}+ x^{\frac 1{2}(\sigma^*-\sigma_0)}\sum_{\begin{subarray}{c}
\rho=\beta+i\gamma \\ \beta< \frac 1{2}(\sigma^*+\sigma_0)\end{subarray}}|s-\rho|^{-2}\\
& \ll \frac{x\log T}{h}+x^{\frac 1{2}(\sigma^*-\sigma_0)}\log T.
\end{align*}
Consequently,
\begin{equation}\label{eq: 2.E}
E_k\ll_{N,\sigma_0}  \Big(\frac {x^2(\log T)^2}{h^2(\log x)^2}+x^{\sigma^*-\sigma_0}(\log T)^2\Big)\int_{D_{T,H}}L_Q(\boldsymbol \gamma(t)-\boldsymbol \theta) dt.  
\end{equation}

From the estimations \eqref{eq: 2.A}, \eqref{eq: 2.B}, \eqref{eq: 2.C}, \eqref{eq: 2.D} and \eqref{eq: 2.E}, we now obtain an estimate of the integral $I_k$:
\begin{align}\label{eq: 2.Z.1}
I_k&\ll_{N,\sigma_0} HQ^{1-2\sigma_0}(\log Q)^{2k}+x^4(\log x)^{2k+1}\exp\big((3+M)Q\big)+\notag\\
&\qquad\quad+\bigg(
Q^{1-2\sigma_0}(\log Q)^{2k-2}+
\frac{x^2}{T^4(\log x)^2}+\notag\\
&\qquad\qquad\quad\qquad+\frac {x^2(\log T)^2}{h^2(\log x)^2}+x^{\sigma^*-\sigma_0}(\log T)^2
\bigg)\int_{D_{T,H}}L_Q(\boldsymbol \gamma(t)-\boldsymbol \theta) dt.  
\end{align}

In the case $k=0$, it is necessary to make the obvious change in the definition of $A_k,\dots,E_k$ in accordance with \eqref{eq: 2.3}.
An analogous argument, applied to \eqref{eq: 2.3}, shows that \eqref{eq: 2.Z.1} holds when $k=0$.

We now turn to find a lower bound for 
\[\int_{D_{T,H}}L_Q(\boldsymbol \gamma(t)-\boldsymbol \theta) dt.\]
Since $L_Q$ is nonnegative, we have
\begin{align*}
\int_{D_{T,H}}L_Q(\boldsymbol \gamma(t)-\boldsymbol \theta) dt\ge \int_T^{T+H}L_Q(\boldsymbol \gamma(t)-\boldsymbol \theta) dt-2h\Big(\sum_{\boldsymbol{n}}|\beta_{\boldsymbol n}|\Big)N(\sigma^*;T;H).
\end{align*}

For the first term on the right, we use Proposition \ref{thm: 2}; for the second one, we use the assumption \eqref{eq: 3.N}. Exactly here the lower bound for the length of our short intervals is determined. 

By \eqref{eq: 1.3} and \eqref{eq: 3.N}, we thus get 
\[
\Big(\sum_{\boldsymbol{n}}|\beta_{\boldsymbol n}|\Big)N(\sigma^*;T,H)\ll \exp(3Q)H^{1-\Delta(\sigma_0)}(\log T)^{\eta(\sigma^*)},
\]
where $\Delta(\sigma_0)\coloneqq 1-\omega(\sigma^*)$.
Taking $h=H^{\frac{\Delta(\sigma_0)}4}$, it follows that 
\[
h\Big(\sum_{\boldsymbol{n}}|\beta_{\boldsymbol n}|\Big)N(\sigma^*;T,H)\ll_{\sigma_0} Q^{-2}H^{-\frac{\Delta(\sigma_0)}{16}}
\]
provided that 
\begin{equation}\label{eq: 2.T.1}
T\gg_{\sigma_0} \big(Q^2\exp((3+M)Q)\big)^{\frac{8}{\nu \Delta(\sigma_0)}}.
\end{equation}
And choosing suitable constants, this, finally, yields
\begin{equation}\label{eq: 2.L}
\int_{D_{T,H}}L_Q\big(\boldsymbol{\gamma}(t)-\boldsymbol{\theta}\big)dt\ge \frac H2.
\end{equation}

Now, we assume that \eqref{eq: 2.T.1} holds. In the combination of \eqref{eq: 2.Z.1} and \eqref{eq: 2.L}, we obtain that  
\begin{align*}
I_k\ll_{N,\sigma_0} \bigg(
Q^{1-2\sigma_0}(\log Q)^{2k}+
\frac{x^4(\log x)^{2k+1}\exp((3+M)Q)}H+\qquad\qquad\qquad\qquad
\notag\\
+\frac {x^2(\log T)^2}{H^{\frac{\Delta(\sigma_0)}2}(\log x)^2}+x^{\sigma^*-\sigma_0}(\log T)^2
\bigg)\int_{D_{T,H}}L_Q(\boldsymbol \gamma(t)-\boldsymbol \theta) dt.  
\end{align*}
Taking $x=H^{\frac{\Delta(\sigma_0)}{8}}\le H^{\frac 18}$ and $T\gg_{\sigma_0} \big(Q^2\exp((3+M)Q)\big)^{\frac{8}{\nu \Delta(\sigma_0)(\sigma_0-\sigma^*)}}$, we have
\[\frac {x^2(\log T)^2}{H^{\frac{\Delta(\sigma_0)}2}(\log x)^2}+x^{\sigma^*-\sigma_0}(\log T)^2\ll_{\sigma_0,\nu} H^{-\frac{\Delta(\sigma_0)}4}+H^{-\frac{\Delta(\sigma_0)(\sigma_0-\sigma^*)}4}\ll_{\sigma_0,\nu}Q^{-2}\]
and
\[\frac{x^4(\log x)^{2k+1}\exp((3+M)Q)}H\ll_N \frac{\exp((3+M)Q)}{H^{\frac{1}4}}\ll_{N,\sigma_0,\nu} Q^{-2}.\]
Choosing $M=\exp(cQ)\gg Q^2\exp(3Q)$ for some absolute constant $c$ and  
\[Q\gg_{N,\sigma_0,\nu} \Big(\frac 1\varepsilon\Big)^{\frac{2}{\sigma_0-\frac 12}},\]
we now obtain
\begin{equation}\label{eq: bound Ik}
I_k\ll_{N,\sigma_0,\nu}\varepsilon^2\int_{D_{T,H}}L_Q(\boldsymbol \gamma(t)-\boldsymbol \theta) dt.
\end{equation}

By Proposition \ref{thm: 1.1}, for $Q$ satisfying \eqref{eq: thm: 1},
there exists $\boldsymbol{\theta}_0=(\theta_2^{(0)},\dots,\theta_p^{(0)},\dots)$ such that
\[
\bigg|\frac{d^k}{ds^k}\log \zeta_{\widehat Q}(s,\boldsymbol{\theta}_0)\Big|_{s=\sigma_0}-a_k\bigg|<\frac \varepsilon 4
\]
for $k=0,1,\dots,N-1$. Note that, for $|\theta_p^{(0)}-\theta_p|<\delta=Q^{-1}$, we have
\begin{align*}
\bigg|\frac{\partial^k}{\partial s^k}\log \zeta_{\widehat Q}(s,\boldsymbol \theta)\Big|_{s=\sigma_0}-\frac{d^k}{d s^k}\log \zeta_{\widehat Q}(s,\boldsymbol \theta_0)&\Big|_{s=\sigma_0}\bigg|\\
& \ll \sum_{p\in \widehat Q}\frac{(\log p)^k}{p^{\sigma_0}}\delta
\ll
Q^{-\sigma_0}(\log Q)^{k-1}
\end{align*}
by Taylor's expansion, partial summation, and the prime number theorem.
If
\begin{equation}\label{eq: 4.1}
|\theta_p^{(0)}-\theta_p|<\delta
\end{equation}
for all $p\le Q$, then 
\begin{equation}\label{eq: 4.2}
\bigg|\frac{\partial^k}{\partial s^k}\log \zeta_{\widehat Q}(s,\boldsymbol{\theta})\Big|_{s=\sigma_0}-a_k\bigg|< \frac \varepsilon 2
\end{equation}
for $k=0,\ldots,N-1$.
Since $L_Q$ is positive, by \eqref{eq: 2.L} and \eqref{eq: bound Ik} for $\boldsymbol{\theta}=\boldsymbol{\theta}_0$, it is not difficult to obtain that
there exist effective positive computable constants $c_2(N,\sigma_0,\nu)$ and $c_3(\sigma_0,\nu)$ such that if 
\begin{equation*}\label{eq: thm 3.Q}
Q\ge c_2(N,\sigma_0,\nu)\Big(\frac{1}{\varepsilon}\Big)^{\frac{2}{\sigma_0-\frac 12}}
\end{equation*}
and
\begin{equation*}\label{eq: thm 3.T}
T\ge \exp_2\big(c_3(\sigma_0,\nu)Q\big),
\end{equation*}
then there exists $\tau\in D_{T,H}\subseteq [T,T+H]$, where $T^{\nu}\le H\le T$ such that
\begin{equation}\label{eq: 4.3}
\bigg|\frac{d^k}{d s^k}\log \zeta(s)\Big|_{s=\sigma_0+ i\tau}-\frac{d^k}{ds^k}\log \zeta_{\widehat Q}(s+ i\tau,\boldsymbol 0)\Big|_{s=\sigma_0}\bigg|<\frac \varepsilon2
\end{equation}
for $k=0,\ldots,N-1$, and 
\begin{equation}\label{eq: 4.4}
L_Q(\boldsymbol \gamma(\tau)-\boldsymbol{\theta}_0)\ne 0.
\end{equation}
By \eqref{eq: 4.4} together with the definition of $L_Q$, the curve $\boldsymbol\gamma(\tau)$ comes close to $\boldsymbol{\theta}_0$. Thus, \eqref{eq: 4.1} holds for $\boldsymbol{\theta}=\boldsymbol{\gamma}(\tau)$. Hence, we also obtain \eqref{eq: 4.2} for $\boldsymbol{\theta}=\boldsymbol{\gamma}(\tau)$, i.e., for $k=0,\ldots,N-1$,
\[
\bigg|\frac{d^k}{ds^k}\log \zeta_{\widehat Q}(s,\boldsymbol{\gamma}(\tau))\Big|_{s=\sigma_0}-a_k\bigg|< \frac \varepsilon 2,
\]
By \eqref{eq: 11} and \eqref{eq: 4.3}, we also obtain for $k=0,\ldots,N-1$,
\[
\bigg|\frac{d^k}{d s^k}\log \zeta(s)\Big|_{s=\sigma_0+ i\tau}-\frac{d^k}{ds^k}\log \zeta_{\widehat Q}(s,\boldsymbol \gamma(\tau))\Big|_{s=\sigma_0}\bigg|<\frac \varepsilon2.
\]
Consequently, for $k=0,\ldots,N-1$,
\[
\bigg|\frac{d^k}{d s^k}\log \zeta(s)\Big|_{s=\sigma_0+ i\tau}-a_k\bigg|<\varepsilon.
\]

Finally, we apply a density theorem due to Balasubramanian \cite{balasubramanian1978} (see \eqref{eq: bal}).
It is easy to see that for any $\alpha\in (\frac 12,1)$,
\[\frac{4(1-\alpha)}{3-2\alpha}=2-\frac{2}{3-2\alpha}\in (0,1).\]
This satisfies the assumption \eqref{eq: 3.N}, and hence, we immediately obtain Theorem~\ref{Thm 1}.

For some parts of the proof of Theorem \ref{Thm 1}, reading the paper \cite{thanasis} was rather useful. The ideas therein also allow to establish a positive proportion of these solutions.

\section{Proof of Theorem \ref{Thm 2}}
The {\it weak} universality result for the Riemann zeta-function was first published in \cite{glmss}. Recently, the result for other functions for long intervals has been done by Kenta Endo \cite{endo2023} (which also contains a small correction of a small inaccuracy in \cite{glmss}); here we derive the result of the Riemann zeta-function which is restricted to short intervals. 

First of all, we provide a version of Theorem \ref{Thm 1} for the Riemann zeta-function in place of its logarithm. The proof follows easily from the original paper.
\begin{theorem}\label{Thm 3}
Let $N\in \mathbb{N}$, $\sigma_0\in (\frac 12,1),\boldsymbol{b}=(b_0,b_1,\dots,b_{N-1})\in \mathbb C^N$ with $b_0\ne 0$ and $\varepsilon\in (0,1)$ be arbitrary but fixed. Then the system of inequalities
\[
\bigg|\frac{d^k}{d s^k}\zeta(s)\Big|_{s=\sigma_0+ i\tau}-b_k\bigg|<\varepsilon \qquad\mbox{for}\quad k=0,1,\ldots,N-1
\]
has a solution $\tau\in [T,T+H]$ provided that $T^{\frac{27}{82}}\le H\le T$ and
\[
T\ge \exp_2\bigg(C_{3}(N,\sigma_0)\Big(|\log b_0|+\frac{1+|b_0|}{\varepsilon}\Big(\frac{\|\boldsymbol{b}\|}{|b_0|}\Big)^{(N-1)^2}\Big)^{\frac 8{1-\sigma_0}+{\frac{8}{\sigma_0-{\frac 1 2}}}}\bigg),
\]
where $C_3(N,\sigma_0)$ is a positive, effectively computable constant depending only on $N$, $\sigma$ and $\|\boldsymbol{b}\|\coloneqq \sum_{0\le k\le N-1}|b_k|$.
\end{theorem}

We now come to the proof of Theorem \ref{Thm 2} which involves the Taylor's expansion of $g$ and Theorem \ref{Thm 3}. By the Cauchy integral formula
\[\frac{d^{k}}{d s^{k}}g(s)\Big|_{s=s_0}=
\frac{k!}{2\pi i}\oint_{|z-s_0|=r}\frac{g(z)}{(z-s_0)^{k+1}} dz,\]
we obtain that
\[\bigg|\frac{d^{k}}{d s^{k}}g(s)\Big|_{s=s_0}(s-s_0)^k\bigg|\le k! M_g \delta_0^k\]
for $|s-s_0|\le \delta_0r$, where $M_g\coloneqq \max_{|s-s_0|=r}|g(s)|$. Hence, by Taylor's expansion, we have
\begin{equation*}
\bigg|g(s)-\sum_{0\le k<N}\frac{d^{k}}{d s^{k}}g(s)\Big|_{s=s_0}\frac{(s-s_0)^k}{k!}\bigg|\le M_g\sum_{k\ge N}\delta_0^k=M_g\frac{\delta_0^N}{1-\delta_0}
\end{equation*}
for $|s-s_0|\le \delta_0r$. We choose $N$ (depending on $g,\delta_0$) for which
\[M_g\frac{\delta_0^{N}}{1-\delta_0}<\frac{\varepsilon}3.\]
Thus, for $|s-s_0|\le \delta_0r$, we have
\begin{equation}
\bigg|g(s)-\sum_{0\le k<N}\frac{d^{k}}{d s^{k}}g(s)\Big|_{s=s_0}\frac{(s-s_0)^k}{k!}\bigg|<\frac \varepsilon3. \label{eq: 9.1}
\end{equation}

Next, by Theorem \ref{Thm 3}, for any $\delta_1\in (0,1)$,
there exists a number $t_1\in [T,T+H]$ such that
\[
\Bigg|\frac{d^{k}}{d s^{k}}\zeta(s)\big|_{s=\sigma_0+it_1}-\frac{d^{k}}{d s^{k}}g(s)\big|_{s=s_0}\Bigg|<\delta_1
\]
for $k=0,1,\dots,N-1$ and 
\[
T\ge \exp_2\bigg(C_{2}(N,\sigma_0)\Big(|\log g(s_0)|+\frac{1+|g(s_0)|}{\delta_1}\Big(\frac{\|\boldsymbol{G}\|}{|g(s_0)|}\Big)^{(N-1)^2}\Big)^{{\frac{8}{1-\sigma_0}}+{\frac{8}{\sigma_0-{\frac 1 2}}}}\bigg),
\]
where
$\|\boldsymbol{G}\|\coloneqq\sum_{0\le k\le N-1}\big|\frac{d^{k}}{d s^k}g(s)|_{s=s_0}\big|$.
Put $\tau\coloneqq t_1-t_0$. Then $\sigma_0+i\tau=s_0+it_1$. Hence,
\begin{align}
\bigg|\sum_{0\le k<N}\frac{d^{k}}{d s^{k}}\zeta(s)\Big|_{s=s_0+i\tau}\frac{(s-s_0)^k}{k!}-\sum_{0\le k<N}\frac{d^{k}}{d s^{k}}g(s)\Big|_{s=s_0}\frac{(s-s_0)^k}{k!}\bigg|\qquad\qquad\notag\\<\delta_1\sum_{0\le k<N}\frac{(\delta_0r)^k}{k!}=\frac\varepsilon 3 \label{eq: 9.2}
\end{align}
provided that $|s-s_0|\le \delta_0r$ and
$\delta_1\coloneqq \frac{\varepsilon}3\exp(-\delta_0r)\in (0,1)$.

Finally, we use the Taylor expansion (again) for $\zeta(s)$ on the shifted disk $\mathcal K+i\tau$. Here, we need to exclude all poles. Hence, we assume that $T>r$.
By the Cauchy integral formula (again)
\[\frac{d^{k}}{d s^{k}}\zeta(s+i\tau)\Big|_{s=s_0}=
\frac{k!}{2\pi i}\oint_{|z-s_0|=r}\frac{\zeta(z+i\tau)}{(z-s_0)^{k+1}} dz,\]
we obtain that
\[\bigg|\frac{d^{k}}{d s^{k}}\zeta(s+i\tau)\Big|_{s=s_0}(s-s_0)^k\bigg|\le k! M_{\zeta}(\tau) \delta^k\]
for $|s-s_0|\le \delta r$ and $0\le \delta\le \delta_0$, where $M_{\zeta}(\tau)\coloneqq \max_{|s-s_0|=r}|\zeta(s+i\tau)|$.  By Taylor's expansion, we have
\begin{equation*}
\bigg|\zeta(s+i\tau)-\sum_{0\le k<N}\frac{d^{k}}{d s^{k}}\zeta(s+i\tau)\Big|_{s=s_0}\frac{(s-s_0)^k}{k!}\bigg|\le M_\zeta(\tau)\frac{\delta^N}{1-\delta}
\end{equation*}
for $|s-s_0|\le \delta r$. Choosing~$\delta$ for which
\[M_\zeta(\tau)\frac{\delta^N}{1-\delta}<\frac{\varepsilon}3,\]
we have
\begin{equation}
\bigg|\zeta(s+i\tau)-\sum_{0\le k<N}\frac{d^{k}}{d s^{k}}\zeta(s)\Big|_{s=s_0+i\tau}\frac{(s-s_0)^k}{k!}\bigg|<\frac{\varepsilon}3. \label{eq: 9.3}
\end{equation}
Combining \eqref{eq: 9.1}, \eqref{eq: 9.2} and \eqref{eq: 9.3} yields Theorem \ref{Thm 2}.

\section{Concluding Remarks}

Since the domain of approximation is restricted and there is no information about the size of the set of shifts, Theorem \ref{Thm 2} is considered a {\it weak} universality theorem for short intervals. The more advanced though ineffective case of universality in the sense of Voronin's result from \cite{vor} was considered first by \cite{laurincikas} with respect to short intervals, and in a recent note \cite{short} conditional and unconditional improvements were given. 

The results in \cite{short} and this note focus on different points. More precisely, the results in \cite{short} focus on an ineffective case based on the bounded mean-square for $\zeta(s)$ together with the result of Bourgain and Watt \cite{BourWatt2017} (and the method of exponent pairs) for unconditional results; the result of Sankaranarayanan and Srinivas \cite{SanSri1993} for conditional results. Meanwhile, those in this note focus on an effective case based on multidimensional $\Omega$-result of Voronin. 

Observe that the proof of Theorem \ref{Thm 1} deals with hypothetical zeros to the right and this proof should be much easier if there are no such zeros. It is reasonable to study the same under conditional assumptions such as the unproven {\it Riemann hypothesis}, which implies that there are no zeros to the right of the critical line, i.e. $N(\sigma^*;T,H)=0$ for all $\sigma^*>\frac 12$. Thus, under the {\it Riemann hypothesis}, our results for short interval $[T,T+H]$ still hold but the length of intervals becomes for every $H$ satisfying \[T^\nu\le H\le T\]
(for any fixed $\nu>0$) and its effective constants depend on parameters $N,\sigma_0$, and~$\nu$. 

On this note, the constants in the main results are effectively computable. Explicit bounds, however, are not too easy to establish. The probably hardest work would be to make the error term in Huxley's prime number theorem \eqref{eq: 06} explicit.

Finally, this can also be done for elements of the Selberg class but needs some slightly more advanced details; for long intervals, this has recently been done by Kenta Endo \cite{endo2023}.

\section*{Acknowledgement}
We are grateful to anonymous referees for carefully reading the manuscript and pointing out the error in an earlier version of this manuscript.

\end{document}